\documentclass[a4paper,11pt]{article}
\date{}
\usepackage{amsmath,amsfonts,enumerate,amssymb,amsthm,color}
\usepackage{pgf,tikz,cite}
\usepackage{fullpage,enumerate}
\usepackage{graphicx}

\newtheorem{thm}{Theorem}
\newtheorem{lem}{Lemma}

\newtheorem{prop}{Proposition}
\newtheorem{cor}{Corollary}

\newtheorem{defi}{Definition}

\begin{document}

\title{Equistarable bipartite graphs
\thanks{This research is supported in part by ``Agencija za raziskovalno dejavnost Republike Slovenije'',
research program P$1$--$0285$ and research projects J$1$-$5433$, J$1$-$6720$, J$1$-$6743$, and BI-US/$14$--$15$--$050$.
The first author also thanks for partial support the National Science Foundation (Grant IIS-1161476).}}

\author{
Endre Boros\\
\small MSIS Department and RUTCOR, Rutgers University, New Jersey, USA\\
\small 100 Rockafeller Rd, Piscataway NJ 08854, USA\\
\small \texttt{Endre.Boros@rutgers.edu}\\
\and
Nina Chiarelli\\
\small University of Primorska, UP FAMNIT, Glagolja\v ska 8, SI6000 Koper, Slovenia\\
\small \texttt{nina.chiarelli@student.upr.si}\\
\and
Martin Milani\v c\\
\small University of Primorska, UP IAM, Muzejski trg 2, SI6000 Koper, Slovenia\\
\small University of Primorska, UP FAMNIT, Glagolja\v ska 8, SI6000 Koper, Slovenia\\
\small IMFM, Jadranska 19, 1000 Ljubljana, Slovenia\\
\small \texttt{martin.milanic@upr.si}
}

\maketitle

\begin{abstract}

\begin{sloppypar}
Recently, Milani\v{c} and Trotignon introduced the class of {\it equistarable} graphs as graphs without isolated vertices admitting
positive weights on the edges such that a subset of edges is of total weight $1$ if and only if
it forms a maximal star. Based on equistarable graphs, counterexamples to three conjectures on equistable graphs were constructed, in particular to Orlin's conjecture, which states that every equistable graph is a general partition graph.
\end{sloppypar}

In this paper we characterize equistarable bipartite graphs.
We show that a bipartite graph is equistarable if and only if every $2$-matching of the graph extends to a
matching covering all vertices of degree at least $2$. As a consequence of this result, we obtain that Orlin's conjecture holds
within the class of complements of line graphs of bipartite graphs.

We also connect equistarable graphs to the triangle condition, a combinatorial condition known to be necessary (but in general not  sufficient) for equistability.  We show that the triangle condition implies general partitionability for complements of line graphs of forests, and
construct an infinite family of triangle non-equistable graphs
within the class of complements of line graphs of bipartite graphs.
\end{abstract}


\section{Introduction}

In~\cite{MT} Milani\v c and Trotignon established a connection between equistarability and equistability and between $2$-extendable graphs and general partition graphs. (See Section~\ref{sec:Preliminaries} for definitions.) In particular, they proved that: (1) if a graph $G$ is triangle-free, then $G$ is equistarable if and only if $\overline{L(G)}$, the complement of its line graph, is equistable and (2) a connected triangle-free graph $G$ of minimum degree at least~$2$ is $2$-extendable if and only if $\overline{L(G)}$ is a general partition graph. Based on this approach, they disproved Orlin's conjecture~\cite{MM15}, which stated that every equistable graph is a general partition graph.
The counterexamples in~\cite{MT} are based on complements of line graphs of triangle-free graphs, and
that work left open the validity of Orlin's conjecture for the class of complements of line
graphs of bipartite graphs, and more generally for the class of perfect graphs. Regarding other subclasses of perfect graphs, Orlin's conjecture can be easily verified to hold for bipartite graphs and their complements, and was shown to hold for chordal graphs~\cite{Pel_Rot} and for line graphs of bipartite graphs~\cite{Lev_Mil}.

We show in this paper that Orlin's conjecture holds for the complements of line graphs of bipartite graphs.
We achieve this by further extending the connections between a triangle-free graph $G$ and the complement of its line graph $\overline{L(G)}$, translating the properties that $\overline{L(G)}$ is a general partition graph, resp.~a triangle graph, to $G$.
We recall that the general partition property and the triangle property are a sufficient and a necessary condition, respectively, for equistability.
We summarize the connections in Table \ref{tabela}.

\begin{table}[h]
\centering
{\renewcommand{\arraystretch}{1.25}
\begin{tabular}{|c|l|c|}
\multicolumn{1}{c}{property of a triangle-free graph $G$} & \multicolumn{1}{l}{} & \multicolumn{1}{c}{corresponding property of $\overline{L(G)}$}
\tabularnewline
\cline{1-1} \cline{3-3}
\begin{tabular}[c]{@{}c@{}}every connected component of $G$ is\\either a star or $2$-internally extendable
\end{tabular}& $\Longleftrightarrow$ & $\overline{L(G)}$  is general partition\tabularnewline
\cline{1-1} \cline{3-3}
\multicolumn{1}{c}{$\Downarrow$} & \multicolumn{1}{l}{} & \multicolumn{1}{c}{$\Downarrow$}\tabularnewline
\cline{1-1} \cline{3-3}
$G$ is strongly equistarable & $\Longleftrightarrow$ & $\overline{L(G)}$  is strongly equistable\tabularnewline
\cline{1-1} \cline{3-3}
\multicolumn{1}{c}{$\Downarrow$} & \multicolumn{1}{l}{} & \multicolumn{1}{c}{$\Downarrow$}\tabularnewline
\cline{1-1} \cline{3-3}
$G$ is equistarable & $\Longleftrightarrow$ & $\overline{L(G)}$  is equistable\tabularnewline
\cline{1-1} \cline{3-3}
\multicolumn{1}{c}{$\Downarrow$} & \multicolumn{1}{l}{} & \multicolumn{1}{c}{$\Downarrow$}\tabularnewline
\cline{1-1} \cline{3-3}
$G$ is $P_{5}$-constrained & $\Longleftrightarrow$ & $\overline{L(G)}$ is triangle\tabularnewline
\cline{1-1} \cline{3-3}
\end{tabular}
}
\caption{Connections between properties of a triangle-free graph $G$ and $\overline{L(G)}$.}\label{tabela}
\end{table}

We show that, when restricted to the class of bipartite graphs, the upper three classes on the left in Table \ref{tabela} coincide. Moreover, when restricted to the class of forests, all four classes on the left in Table \ref{tabela} coincide.

The paper is structured as follows. Section \ref{sec:Preliminaries} contains the basic definitions and known lemmas
that establish the known equivalences and implications in Table \ref{tabela}. In Section \ref{sec:Basic results} we establish the first and the last equivalence from Table \ref{tabela} and observe that the implications in the left side of the table cannot be reversed.
Sections \ref{sec:Equistarable bipartite} and \ref{sec:forests} contain our main results, that is, a complete characterization of equistarable bipartite graphs and of equistarable forests, along with some algorithmic aspects concerning the recognition of these newly characterized families.

\section{Preliminaries}\label{sec:Preliminaries}

All graphs in the paper will be finite, simple and undirected. For undefined graph theoretic notions, we refer to \cite{west}.
 A {\it stable} or ({\it independent}) set in a graph is a set of pairwise non-adjacent vertices; a stable set is said to be {\it maximal} if it is not contained
 in any other stable set. A {\it clique} in a graph is a set of pairwise adjacent vertices.
We denote by $N(u)$ the set of all neighbors of $u$ and $N(U)=(\cup_{u\in U}N(u))\setminus U$. The {\it degree} of a vertex $u$ in a graph $G$, denoted by $d_G(u)$, is equal to $|N(u)|$. The minimum degree of a graph $G$ is the minimum degree of its vertices and is denoted by $\delta(G)$.

The {\it complement of a graph} $G$ is the graph $\overline{G}$ with the same vertex set as $G$ in which two distinct vertices are adjacent if and only if they are not adjacent in $G$. The {\it line graph} $L(G)$ of $G$ is a graph such that:
$(i)$ the vertex set of $L(G)$ is the edge set of $G$ and
$(ii)$ two distinct vertices of $L(G)$ are adjacent if and only if they share a common endpoint as edges in $G$.
A graph $G$ is {\it bipartite} if its vertex set can be partitioned into two independent sets, and
 {\it triangle-free} if it does not have a triangle ($K_3$) as induced subgraph.
A graph is said to be a {\it star} if it is isomorphic to the complete bipartite graph $K_{1,n}$ for some $n\geq 1$.

A graph $G=(V,E)$ is a {\it general partition graph} if there exists a set $U$ and an assignment of non-empty subsets $U_x\subseteq U$ to the vertices of $G$ such that two vertices $x$ and $y$ are adjacent if and only if $U_x \cap U_y \neq \emptyset$ and for every maximal stable set $S$ of $G$,
the set $\{U_x\, : x \in S\}$ is a partition of $U$~\cite{McAvaney14}.
A graph $G=(V,E)$ is said to be {\it equistable} if and only if there exists a mapping $\varphi : V \rightarrow \mathbb{R}_+$ such that
for all $S \subseteq V$, $S$ is a maximal stable set in $G$ if and only if $\varphi (S):= \sum_{v\in S} \varphi (v) = 1$~\cite{MR553649}.
In 1994 Mahadev et al.~introduced in~\cite{MPS13} a subclass of equistable graphs, the so-called strongly equistable graphs.
For a graph $G$, we denote by $\mathcal{S}(G)$ the set of all maximal stable sets of $G$, and by $\mathcal{T}(G)$ the set of all other nonempty subsets of $V(G)$. A graph is said to be {\it strongly equistable} if for each $T \in \mathcal{T}(G)$ and for each $\gamma \leq 1$ there exists a function  $\varphi:V \rightarrow \mathbb{R}_+$ such that $\varphi(S)=1$ for all $S\in \mathcal{S}(G)$, and $\varphi(T)\neq \gamma$.
The {\it triangle condition} was introduced by McAvaney et al.~in \cite{McAvaney14} and states that for every maximal stable set $S$ in $G = (V, E)$ and every edge $uv$ in $G - S$ there is a vertex $s \in S$ such that $\{u, v, s\}$ induces a triangle in $G$. Graphs satisfying this condition are
called {\it triangle graphs}.

A {\it matching} in a graph $G$ is a set of pairwise disjoint edges. Given a matching $M$ in a graph $G$, we say that a vertex $v$ is {\it covered} (or {\it saturated}) by $M$ if $v$ is an endpoint of an edge in $M$. We will denote by $V(M)$ the set of all endpoints of edges in $M$. Given two matchings $M$ and $M'$ in a graph $G$, we say that $M$ {\it extends to} $M'$ if $M\subseteq M'$.
A matching consisting of exactly $k$ edges will be referred to as a {\it $k$-matching} (where $k$ is the {\it size} of the matching).
A matching $M$ is said to be a {\it perfect matching} of $G$ if it covers all vertices of $G$.
We say that a matching $M$ in a graph $G$ is a {\it perfect internal matching} if every vertex not covered by $M$ is a {\it leaf}, that is, a vertex of degree 1. Perfect internal matchings were studied in a series of papers, see for example \cite{Bar_Gombas,Bar_Mik_ARS,Bar_Mik_IPL} and references cited therein. For a general reference on matching theory, see~\cite{LovPlum}.

A connected graph $G$ is said to be {\it $k$-extendable} if $G$ contains a $k$-matching and every $k$-matching can be extended to a perfect matching~\cite{MR583220}. We generalize this notion as follows.

\begin{defi}\label{def:k-internally ext.}
Given $k\geq 1$, a connected graph $G$ is said to be {\upshape $k$-internally extendable} if $G$ contains a $k$-matching and every $k$-matching extends to a perfect internal matching.
\end{defi}

Even though for the purposes of stating the main results in this paper (Theorems~\ref{thm:main_for_bipartite} and~\ref{thm:main_for_forests}),
it would be more convenient to define a graph $G$ to be $k$-internally extendable simply as a graph in which every $k$-matching extends to a perfect internal matching, we decided to keep the definition more restrictive, requiring also connectedness and the existence of a $k$-matching. This is because this way, the definition is similar to the definition of $k$-extendable graphs; moreover, the two notions coincide for graphs of minimum degree at least $2$. In this paper, we consider $k$-internally extendable graphs only for $k\in \{1,2\}$.

Given a graph $G$ and a vertex $v \in V(G)$, the {\it star rooted at $v$} is the set $E(v)$ of all edges incident with $v$. A {\it star} of $G$ is a star rooted at some vertex $v \in V(G)$ and a star is said to be {\it maximal} if it is not properly contained in any other star. We denote the union of stars from a set of vertices as $E(U):= \cup_{u\in U}E(u)$, where $U\subseteq V(G)$.
In \cite{MT} {\it equistarable graphs} were introduced as graphs $G=(V,E)$ without isolated vertices for which there exist a mapping
$\varphi : E \rightarrow \mathbb{R}^+$ on the edges of $G$ such that a subset $F \subseteq E$ is a maximal star in $G$ if and only if
$\varphi (F) := \sum_{e \in F} \varphi (e) = 1$. Such a mapping $\varphi$ is called an {\it equistarable weight function} of $G$.
Note that for every equistarable weight function $\varphi$, we have $\varphi(e) > 0$ for all $e\in E(G)$, since otherwise
that would directly imply that we would have a total weight of $1$ on some subset of the edges
that does not induce a maximal star in $G$. For a graph $G$ we denote by $\mathcal{S}^*(G)$ the set of all maximal stars of $G$, and by $\mathcal{T}^*(G)$ the set of all other nonempty subsets of $E(G)$. A graph $G=(V,E)$ without isolated vertices is said to be {\it strongly equistarable} if for each $T \in \mathcal{T}^*(G)$ and each $\gamma \leq 1$ there exists a mapping $\varphi : E \rightarrow \mathbb{R}^+$ such that $\varphi (S)=1$ for all $S \in \mathcal{S}^*(G)$, and $\varphi(T)\neq \gamma$ \cite{MT}.
Given a graph $G$ and a subset of edges $F\subseteq E(G)$, the {\it characteristic vector} of $F$ is the vector
$\chi^F\in \{0,1\}^{E(G)}$ defined as $\chi^F_e = 1$ if
$e\in F$, and $\chi^F_e = 0$, otherwise.

We also introduce the following property imposing a constraint on $5$-vertex paths in the graph. A graph $G$
is said to be
{\it $P_5$-constrained} if the middle vertex of every (not necessary induced) $5$-vertex path $P_5$ in $G$ is of degree at least $3$,
that is, it is incident with at least one edge not in the path.

We conclude this section by discussing the validity of Table~\ref{tabela}. The following lemmas from \cite{MT} establish the second and third equivalences in Table \ref{tabela}.

\begin{lem}\label{lem:strongly_equistar-strongly_equistab}
Let $G$ be a triangle-free graph. Then $G$ is strongly equistarable if and only if  $\overline{L(G)}$ is a strongly equistable graph.
\end{lem}

\begin{lem}\label{lem:equistar-equistab}
Let $G$ be a triangle-free graph. Then $G$ is equistarable if and only if $\overline{L(G)}$ is equistable.
\end{lem}

The first and the third implication in the right side of Table~\ref{tabela} were proved by Miklavi\v c and Milani\v c in~\cite{MM15} (the third one was
essentially observed already in~\cite{MPS13}), and, as already mentioned, the second implication in the right side of Table~\ref{tabela} was proved by Mahadev et al.~in~\cite{MPS13}. Lemmas \ref{lem:strongly_equistar-strongly_equistab} and~\ref{lem:equistar-equistab} along with the
second implication in the right side of Table~\ref{tabela} directly imply the second implication in the left side
of Table~\ref{tabela}.

\begin{cor}
Every strongly equistarable triangle-free graph $G$ is equistarable.
\end{cor}

In fact, a straightforward adaptation of either the geometrical proof of the result of Mahadev-Peled-Sun from \cite{MPS13} or the alternative proof given in \cite{MT} to the setting of strongly equistarable graphs shows that the same statement holds for general graphs (not necessary triangle-free).
That is, every strongly equistarable graph is equistarable.

The first and the fourth equivalence in Table~\ref{tabela} will be proved in Lemmas~\ref{lem:co-line-gpg} and~\ref{lem:co-line-triangle}, respectively, in Section~\ref{sec:Basic results}. In turn, this will imply the validity of the first and the third implication in the left side of
Table~\ref{tabela}.

\section{Basic results and examples}\label{sec:Basic results}

As proved by Korach et al.~in~\cite{Korach-Peled-Rotics}, a graph $G$ is equistable if and only if each connected component of $G$ is equistable.
As we show in Lemma~\ref{lem:equistarable components} below, the analogous property for equistarable graphs holds only in one direction.
To this end, recall that in~\cite{MT}
the following property of the graph $H$ depicted in Fig.~\ref{fig:equistarable_notstrongly}
was observed.

\begin{figure}[h!]
  \centering
   \includegraphics[width=50mm]{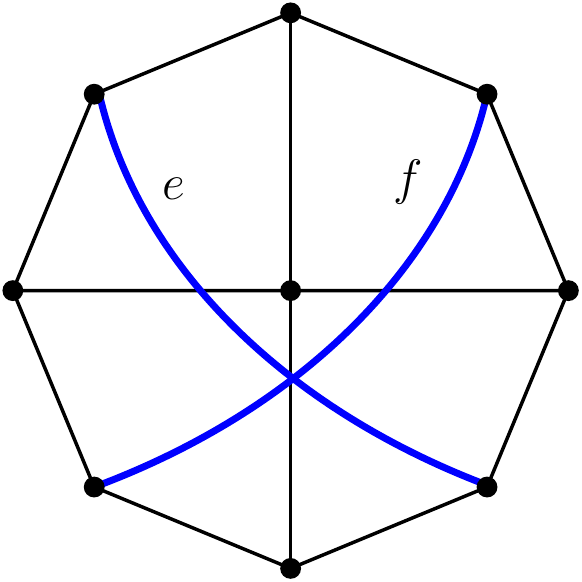}\\
  \caption{A graph $H$ that is equistarable but not strongly equistarable.}\label{fig:equistarable_notstrongly}
\end{figure}

\begin{lem}\label{lem:H}
Let $H$ be the graph depicted in Fig.~\ref{fig:equistarable_notstrongly}, and let $\varphi:E(H)\to \mathbb{R}_+$ be any weight function on the edges of $H$ such that $\varphi(E(v)) = 1$ for every $v\in V(H)$. Then, the $2$-matching $\{e,f\}$ (see Fig.~\ref{fig:equistarable_notstrongly}) satisfies $\varphi(\{e,f\}) = 1/2$.
\end{lem}

The statement of Lemma~\ref{lem:H} appeared within the proof of Proposition~2 in~\cite{MT} (where graph $H$ was named $G^*$) and was established using the structure of a basis of the kernel of the incidence matrix of the graph. We offer an alternative, shorter proof, using a method
that we will apply again later (in the proof of Proposition~\ref{prop:Petersen}).

\begin{proof}[Proof of Lemma \ref{lem:H}]
Let $\lambda:V(H)\to \mathbb{R}$ be the mapping as depicted in Fig.~\ref{fig:H}.

\begin{figure}[h!]
  \centering
   \includegraphics[width=50mm]{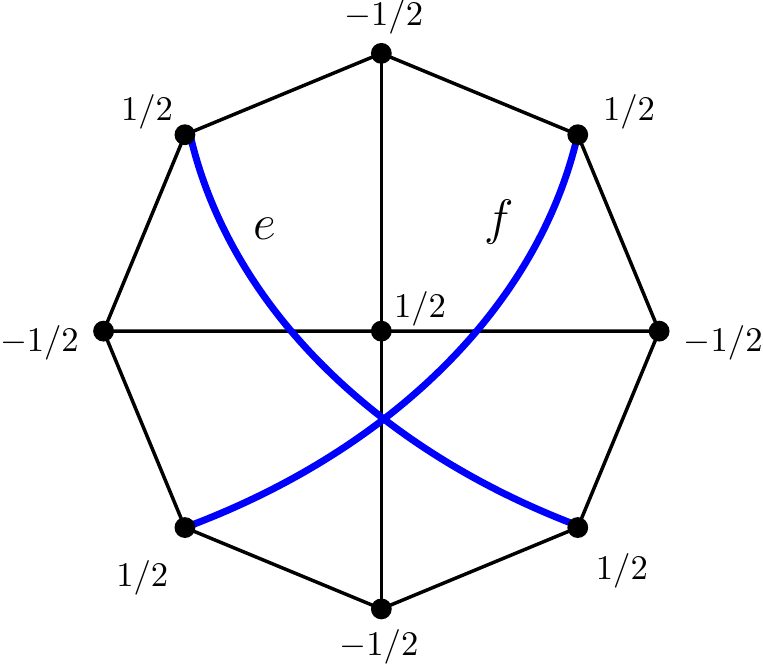}\\
  \caption{The mapping $\lambda:V(H)\to\{-1/2,1/2\}$.}\label{fig:H}
\end{figure}

\noindent That is, vertices in the bigger part of the bipartite graph $H-\{e,f\}$ get assigned weight $1/2$, and all the other vertices get weight $-1/2$.
Using the coefficients given by the mapping $\lambda$, the characteristic vector of the $2$-matching $M = \{e,f\}$ can be expressed as a linear combination of the characteristic vectors of maximal stars, that is,
$\chi^M = \sum_{v\in V(H)}\lambda(v)\cdot \chi^{E(v)}\,.$
Since $\sum_{v\in V(H)}\lambda(v) = 1/2$ and $\varphi(E(v)) = 1$ for all $v\in V(H)$, this implies that $\varphi(M) = 1/2$.
\end{proof}

\begin{lem}\label{lem:equistarable components}
If a graph $G$ is equistarable, then every connected component of $G$ is equistarable.
On the other hand, the class of equistarable graphs is not closed under taking disjoint union.
\end{lem}

\begin{proof}
If $\varphi$ is an equistarable weight function of a graph $G$ and $C$ is a connected component of $G$, then $\varphi'$, the restriction of $\varphi$ to $E(C)$ is an equistarable weight function of $C$. This follows immediately from the definitions, using the fact that for every subset $F$ of $E(C)$,
$F$ is a maximal star of $C$ if and only if it is a maximal star of $G$.

The fact that the class of equistarable graphs is not closed under taking disjoint union can be now justified using Lemma~\ref{lem:H}. Indeed, the lemma implies that the disjoint union of two copies of $H$ contains a $4$-matching of total weight $1$ in each equistarable weight function, and is hence not equistarable.
\end{proof}

The following consequence of Lemma~\ref{lem:equistarable components} is in contrast with the fact that the class of strongly equistable graphs is closed under join, as proved by Mahadev et al.~in~\cite{MPS13}.

\begin{cor}
The set of equistable graphs is not closed under join.
\end{cor}

\begin{proof}
Since the graph $H$ depicted in Fig.~\ref{fig:equistarable_notstrongly} is an equistarable triangle-free graph, the complement of its line graph
is equistable (see Table~\ref{tabela}). Let $2H$ denote the disjoint union of two copies of $H$.
It can be seen that the graph $\overline{L(2H)}$ is isomorphic to the join of two copies of $\overline{L(H)}$.
However, since the graph $2H$ is not equistarable, the graph $\overline{L(H)}$ is not equistable.
\end{proof}

The following lemma is from \cite{MT}.

\begin{lem}\label{lem:gen.par-2ext}
Let $G$ be a triangle-free graph with $\delta(G)\geq 2$. Then, $\overline{L(G)}$ is a general partition graph if and only if every $2$-matching of $G$ extends to a perfect matching.
\end{lem}

We now generalize Lemma \ref{lem:gen.par-2ext}, thus establishing the first equivalence in Table \ref{tabela}.
In the proof we will make use of a result by McAvaney et al., which we now state. A {\it strong clique} in a graph $G$ is a clique
containing at least one vertex (equivalently: exactly one vertex) from
each maximal stable set.

\begin{thm}[McAvaney et al.~\cite{McAvaney14}]\label{thm:strong-cliques}
Let $G$ be a graph. Then, $G$ is a general partition graph if and only if every edge of $G$ belongs to a strong clique.
\end{thm}

\begin{lem}\label{lem:co-line-gpg}
For every triangle-free graph $G$, the following conditions are equivalent:
\begin{enumerate}
  \item $\overline{L(G)}$ is a general partition graph.
  \item Every $2$-matching of $G$ extends to a perfect internal matching.
  \item Every component of $G$ is either a star or $2$-internally extendable.
\end{enumerate}
\end{lem}

\begin{proof}
We first show the equivalence of conditions 1 and 2, and then the equivalence of conditions 2 and 3.

For $1\Leftrightarrow 2$, we first use Theorem~\ref{thm:strong-cliques}
to infer that it suffices to show that
every edge of  $\overline{L(G)}$ belongs to a strong clique
if and only if
every $2$-matching of $G$ extends to a perfect internal matching.
It follows from the definitions of the complement and the line graph operators that
a subset $M\subseteq E(G)$ is a $2$-matching of $G$ if and only if it forms an edge of $\overline{L(G)}$.
Moreover, a set $F\subseteq E(G) = V(\overline{L(G)})$ forms a strong clique in $\overline{L(G)}$ if and only if
$F$ forms a stable set in $L(G)$ intersecting all maximal cliques of $L(G)$.
Since $G$ is triangle-free, maximal cliques of $L(G)$ correspond bijectively to the maximal stars of $G$.
Therefore, a stable set $F$ in $L(G)$ intersecting all maximal cliques of $L(G)$ is a matching of $G$ intersecting all
maximal stars of $G$. Equivalently, $F$ is a perfect internal matching of $G$ that in addition contains all unique edges of components
of $G$ isomorphic to $K_2$ (since the stars rooted at any vertex of the $K_2$ are maximal stars).

The above implies that every edge of  $\overline{L(G)}$ belongs to a strong clique
if and only if every $2$-matching of $G$ extends to a
perfect internal matching that contains all unique edges of components
of $G$ isomorphic to $K_2$. This last condition is easily seen to be equivalent to
the condition that every $2$-matching of $G$ extends to a
perfect internal matching, and establishes the equivalence  $1\Leftrightarrow 2$.

To prove $2\Rightarrow 3$, suppose that every $2$-matching of $G$ extends to a perfect internal matching, let $C$ be a component of $G$ that is not a star, and let $M$ be a $2$-matching in $G$. By assumption, $M$ extends to a perfect internal matching $M'$. Then,
matching $M'\cap E(C)$ is a perfect internal matching of $C$ extending $M$.

It remains to show $3\Rightarrow 2$. Suppose that every connected component of $G$ is either a star or $2$-internally extendable.
Then, each component $C$ of $G$ contains a perfect internal matching, say $M_C$.
Let $M= \{e,f\}$ be a $2$-matching in $G$.
If $M$ is contained in a single connected component of $G$, say $C$, then
$M$ is contained in a perfect internal matching of $C$, say $M'$, and a perfect internal matching
of $G$ extending $M$ is given by  $M'\cup \bigcup_{C'\in {\cal C}\setminus\{C\}}M_{C'}$, where ${\cal C}$ denotes the set of all connected components of $G$.

Suppose now that the two edges of $M$ belong to different connected components of $G$, say $e\in C_e$ and $f\in C_f$.
We claim that $e$ is contained in an internal perfect matching of $C_e$ (and then by symmetry, $f$ is contained in an internal perfect matching of $C_f$). If $e$ extends to a $2$-matching of $C_e$, then we can apply the assumption that $C_e$ is $2$-internally extendable, and the claim follows.
So suppose that $\{e\}$ is a maximal matching in $C_e$. Then, the set $V(C_e)\setminus \{x,y\}$, where $e = xy$, is an independent set
in $C_f$. Since $C_e$ is triangle-free, every vertex of  $V(C_e)\setminus \{x,y\}$ is of degree~$1$ in $C_e$. But this implies that $\{e\}$ itself  is a perfect internal matching of $C_e$.

Denoting by $M_e$ and $M_f$ internal perfect matchings of $C_e$ and $C_f$ containing $e$ and $f$, respectively,
a perfect internal matching of $G$ extending $M$ is given by  $M_e\cup M_f\cup \bigcup_{C'\in {\cal C}\setminus\{C_e,C_f\}}M_{C'}$.
This establishes the implication $3\Rightarrow 2$ and completes the proof.
\end{proof}

The last equivalence from Table \ref{tabela} is proved in the next lemma.
Recall that a graph is said to be
{\it $P_5$-constrained} if the middle vertex of every (not necessary induced) $5$-vertex path $P_5$ in $G$ is incident with at least
one edge not in the path.

\begin{lem}\label{lem:co-line-triangle}
Let $G$ be a triangle-free graph with at least one edge. Then $\overline{L(G)}$ satisfies the triangle condition if and only if $G$ is $P_5$-constrained.
\end{lem}

\begin{proof}
Suppose first that $\overline{L(G)}$ satisfies the triangle condition but $G$ is not $P_5$-constrained. Therefore there exists a $P_5=(v_1, \dots, v_5)$ in $G$ where $d_G(v_3)=2$. The set $E(v_3)=\{v_2v_3,v_3v_4\}$ is a maximal star in $G$, therefore in $\overline{L(G)}$ it is a maximal stable set, because $G$ is triangle-free. Let us introduce $e = v_1v_2$ and $f = v_4v_5$. Since in the graph $\overline{L(G)}$ the vertex $v_2v_3 \in V(\overline{L(G)})$ is adjacent to $f$ but not to $e$ and $v_3v_4 \in V(\overline{L(G)})$ is adjacent to $e$ but not to $f$, we conclude that the edge $\{e,f\}$ in $\overline{L(G)}$ can not be extended to a triangle with a vertex from the maximal stable set $E(v_3)$, a contradiction with the triangle condition.

Suppose now that $G$ is $P_5$-constrained. We will verify that the triangle condition holds for $\overline{L(G)}$.
Take a maximal stable set $S$ in $\overline{L(G)}$ and a pair of adjacent vertices $e,f \in V(\overline{L(G)})\setminus S$.
Note that $S$ is a maximal star in $G$, centered at some vertex $v$ and $\{e,f\}$ is a $2$-matching in $G$. Note also that since $e,f \notin S$, these two edges are not incident with $v$ in $G$.
It is enough to show that there exists an edge $g\in S\subseteq E(G)$ such that $\{e,f,g\}$ form a $3$-matching in $G$.
Equivalently, we need to show that $v$ has a neighbor not covered by the matching $\{e,f\}$.
If $d_G(v)=1$, that is, $|S| = 1$, then, since $S$ is a maximal star, the connected component
of $G$ containing vertex $v$ is a single edge, $g\in E(G)$. Thus, no edge of $G$ is incident with $g$, and hence
$\{e,f,g\}$ is a $3$-matching. If $d_G(v)= 2$, then, since $G$ is triangle-free and $P_5$-constrained, at most one neighbor of $v$ is covered by $\{e,f\}$, hence in this case
there exists an uncovered neighbor of $v$. Similarly, if $d_G(v)\ge 3$, then the triangle-freeness implies that $e$ and $f$ cover at most two neighbors of $v$, and again there exists an uncovered one.
So in all cases we can extend $\{e,f\}$ to a $3$-matching using an edge incident with $v$.
This shows that  $\overline{L(G)}$ satisfies the triangle condition.
\end{proof}

\subsection{Counterexamples to converses of the implications in Table~\ref{tabela}}

As explained at the end of Section~\ref{sec:Preliminaries}, Lemmas~\ref{lem:co-line-gpg} and~\ref{lem:co-line-triangle} together with previously known results suffice to justify all implications and equivalencies in Table~\ref{tabela}.
In~\cite{MM15}, circulant graphs of the form  $C_{n}(\{1,3\})$ for odd $n\ge 11$ were given as examples of a connected triangle-free strongly equistarable graphs that are not $2$-extendable; thus, since $4$-regular, these graphs are also not $2$-internally extendable. This shows that the first implication on either side of Table~\ref{tabela} cannot be reversed.
The graph $H$ depicted in Fig.~\ref{fig:equistarable_notstrongly} was given in~\cite{MM15} as an example of
an equistarable but not strongly equistarable graph. Note that the graph is triangle-free, and hence shows that
the second implication on either side of Table~\ref{tabela} cannot be reversed.

What about the third implication on either side of Table~\ref{tabela}? Regarding examples of non-equistable graphs satisfying the triangle condition,
all the examples known so far (to us) can be found in \cite{MM15}. These are:
\begin{itemize}
  \item A specific example discovered already by DeTemple et al.~\cite{MR1212882}: a $9$-vertex graph $G^*$ given by $V(G^*)=\{1,\dots,9\}$ and by the family of its maximal stable sets ${\cal S}(G)=\{\{1,2,3\},\{4,5,6\}$, $\{7,8,9\},\{1,4,7\},\{3,6,9\}\}$.
  \item An infinite family consisting of tensor product graphs of the form $K_m\times K_n$ with \hbox{$m>n\ge 3$}.
\end{itemize}
(Recall that the {\it tensor product} of two graphs $G$ and $H$ is the graph $G\times H$ with vertex set $V(G)\times V(H)$ in which two vertices $(u_1,v_1)$ and $(u_2,v_2)$ are adjacent if and only if $u_1v_1\in E(G)$ and $u_2v_2\in E(H)$.)
It is an easy exercise to verify that for each $m,n$, we have
\hbox{$K_m\times K_n \cong \overline{L(K_{m,n})}$}
 where $\cong$ denotes the graph isomorphism relation. Therefore, since the graphs $K_{m,n}$ are triangle-free, Lemma~\ref{lem:co-line-triangle} implies that the graphs $K_{m,n}$ for $m>n\ge 3$ are $P_5$-constrained and non-equistarable.
In the next proposition, we offer a short direct proof of this fact.

\begin{prop}
For every $m > n \geq 3$, the complete bipartite graph $K_{m,n}$ is $P_5$-constrained and not equistarable.
\end{prop}

\begin{proof}
Let $G=(A\cup B, E)=K_{n,m}$ with $3\leq n < m$.
Since $\delta(G)=n\geq 3$ the graph $G$ is $P_5$-constrained.
Suppose $G$ is equistarable and
fix an equistarable weight function $\varphi:E(G)\to \mathbb{R}_+$.
Note that every vertex of $G$ is a center of a maximal star. Hence,
$n = \varphi(E(A)) = \varphi(E) = \varphi(E(B)) = m$,
a contradiction.
\end{proof}

The specific example $G^*$ mentioned above was introduced by DeTemple et al.~in~\cite{MR1212882} as the intersection graph of a system of chords on a circle (represented on the left in Fig.~\ref{fig:system of   chords and bipartite graph}).
\begin{figure}[h!]
  \centering
   \includegraphics[width=100mm]{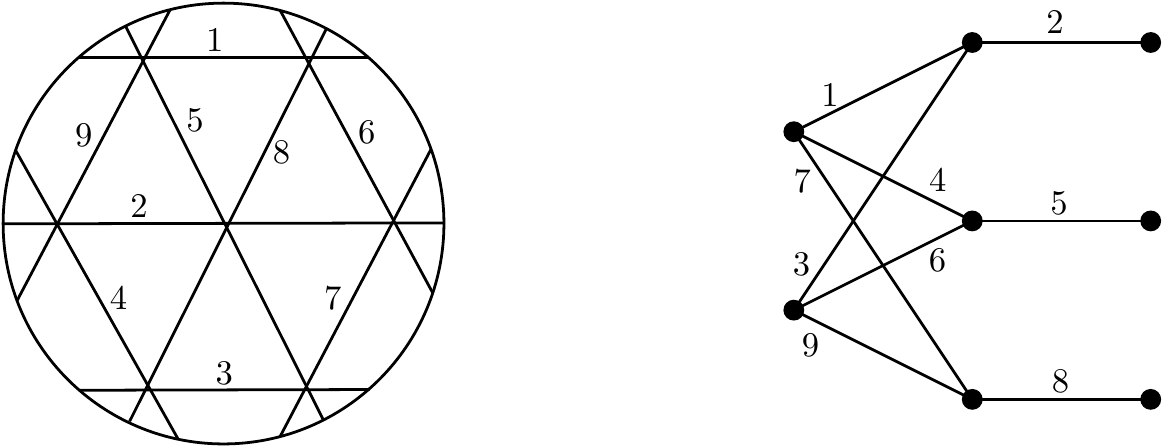}\\
  \caption{Left: a system of chords in a circle, the intersection graph of which is $G^* = \overline{L(K_{2,3}^+)}$; right: $K_{2,3}^+$.}\label{fig:system of chords and bipartite graph}
\end{figure}
It turns out that $G^*$ is also of the form $G^* = \overline{L(G)}$ where $G$ is a
bipartite $P_5$-constrained non-equistarable graph.
In fact, such $G$ is the smallest member of an infinite family ${\cal B}$ of
bipartite $P_5$-constrained non-equistarable graphs, which we define now.
Let $K_{m,n}$ be the complete bipartite graph with a fixed bipartition $\{A,B\}$ of its vertex set with $|A|=m$ and $|B|=n$.
Let $K_{m,n}^+$ be a graph obtained from $K_{m,n}$ by adding to each vertex from set $B$ a private neighbor (a leaf).
Then ${\cal B}=\{K_{m,n}^+ \mid 3\leq n \leq m+1\}$. The graph $K_{2,3}^+$ is shown in Fig.~\ref{fig:system of   chords and bipartite graph}.

\begin{prop}
Each graph in ${\cal B}$ is $P_5$-constrained and not equistarable.
\end{prop}

\begin{proof}
Let $G \in {\cal B}$, with bipartition $\{A',B\}$ such that $A'=A\cup L$, where $L$ is the set of leaves and $|A|= m$, $|B|=|L|= n$.
Since $n\geq 3$, we do not have any vertices of degree $2$, so $G$ is $P_5$-constrained.
Suppose that $G$ is equistarable and fix an equistarable weight function $\varphi:E(G)\to \mathbb{R}_+$.
Note that every vertex of $A\cup B$ is the center of a maximal star.
Summing up the maximal stars from each partition we get
$n = \varphi(E(B)) = \varphi (E(A)) + \varphi (E(L)) = m + \varphi (E(L))$
implying $\varphi (E(L)) = n - m$.
Since $\varphi(e)> 0$ for all $e\in E(G)$, we have $\varphi (E(L)) > 0$ and therefore $n>m$, which together with
$n\le m+1$ implies $n = m+1$ and hence $\varphi (E(L)) = 1$. This is
a contradiction since the set of edges in $E(L)$ does not induce a star in $G$.
\end{proof}

All the above examples of $P_5$-constrained non-equistarable triangle-free graphs are bipartite.
In the next proposition we exhibit a graph that is not of this form.

\begin{prop}\label{prop:Petersen}
The Petersen graph is triangle-free, $P_5$-constrained, and not equistarable.
\end{prop}

\begin{proof}
Let $G$ be the Petersen graph. Clearly, $G$ is triangle-free as well as $P_5$-constrained (since it is $3$-regular).
Suppose for a contradiction that $G$ is equistarable and let $\varphi:E(G)\to \mathbb{R}_+$ be an equistarable weight function of $G$.
Fix a $3$-matching $M$ of $G$ such that the subgraph of $G$ induced by $M$ is $1$-regular (for example, let $M$ consist of the three thick edges as in Fig.~\ref{fig:Petersen}), and let $\lambda:V(G)\to \mathbb{R}$ be the mapping given by
$\lambda(v) = \left\{
                         \begin{array}{ll}
                           1/2, & \hbox{if $v\in V(M)$;} \\
                           -1/2, & \hbox{otherwise.}
                         \end{array}
                       \right.$
\begin{figure}[h!]
  \centering
   \includegraphics[width=60mm]{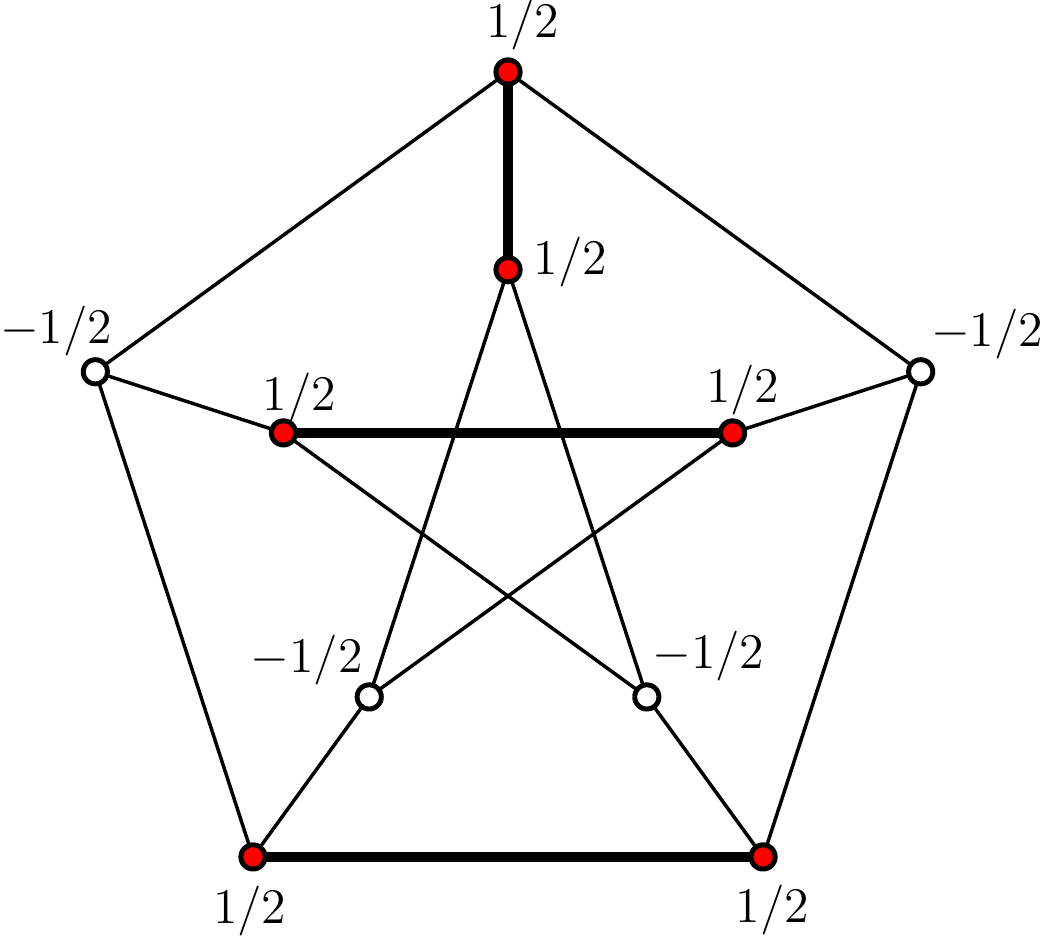}\\
  \caption{The Petersen graph together with a $3$-matching and a $\{1/2, -1/2\}$-weighting of its vertices
  }\label{fig:Petersen}
\end{figure}

To complete the proof, observe that the characteristic vector of $M$ can be expressed
               in the form $$\chi^M = \sum_{v\in V(G)}\lambda(v)\cdot \chi^{E(v)}\,.$$
However, since $\sum_{v\in V(G)}\lambda(v) = 1$, this means that we have expressed
the characteristic vector of $M$
as an affine combination of the characteristic vectors of the (maximal) stars of $G$.
Since $\varphi(E(v)) = 1$ for all $v\in V(G)$, this implies that $\varphi(M) = 1$, contrary to the fact that $\varphi$ is an equistarable weight function of $G$ and $M$ is not a star. This shows that the Petersen graph is not equistarable.
\end{proof}

\begin{cor}
The complement of the line graph of the Petersen graph is a non-equistable graph satisfying the triangle condition.
\end{cor}

\section{Equistarable bipartite graphs}\label{sec:Equistarable bipartite}

When restricted to complements of line graphs of triangle-free graphs of minimum degree at least~$2$, Orlin's conjecture can be rephrased in terms of equistarable graphs as follows: every connected component of an equistarable triangle-free graph $G$ of minimum degree at least $2$ is $2$-extendable. As the graph in Fig.~\ref{fig:equistarable_notstrongly} shows, this is not the case.
The work~\cite{MT} left open the validity of Orlin's conjecture for the class of complements of line
graphs of bipartite graphs, and more generally for the class of perfect graphs.
In this section, we prove that Orlin's conjecture holds for complements of line graphs of bipartite graphs (note that these graphs are perfect), using the notions of $1$-~and $2$-internal extendability (see Definition \ref{def:k-internally ext.}). In particular, we show that in the case of bipartite graphs, the classes of $(i)$ graphs in which each connected component is either a star or $2$-internally extendable, $(ii)$ strongly equistarable graphs, and $(iii)$ equistarable graphs, all coincide (cf.~Table~\ref{tabela}).

\begin{lem}\label{lem:equistarable-1-extendable}
Let $G$ be a connected equistarable bipartite graph with $\delta(G)\ge 2$.
Then, $G$ is $1$-extendable.
\end{lem}

We offer two proofs of this lemma. Our first proof is based on the classical Birkhoff-von Neumann theorem on doubly stochastic matrices.

\begin{proof}[First proof of Lemma \ref{lem:equistarable-1-extendable}]
Fix a bipartition $\{A,B\}$ of $V(G)$, and
an equistarable weight function $\varphi:E(G)\to \mathbb{R}_+$.
Since $\delta(G)\ge 2$, every star is maximal.
It follows that
$|A| = \varphi(E(A)) =\varphi(E(G)) = \varphi(E(B)) = |B|$.
Let $n = |A| = |B|$, and let $Q$ be the
$n\times n$ matrix with rows are indexed over $A$, and columns indexed over $B$
defined by
$$Q_{a,b} = \left\{
             \begin{array}{ll}
\varphi(ab), & \hbox{if $ab\in E(G)$;}\\
               0, & \hbox{otherwise.}
             \end{array}
           \right.$$
For every $a\in A$, we have
$\sum_{b\in B} Q_{a,b} = \sum_{b\in N_G(a)}\varphi(ab) = \varphi(E(a)) = 1$,
and similarly $\sum_{a\in A} Q_{a,b} = 1$ for all $b\in B$.
Since $Q$ only has non-negative entries, it is doubly stochastic.
By the Birkhoff-von Neumann theorem, $Q$ can be written as a convex combination of permutation matrices,
say
$Q = \sum_{i = 1}^k\lambda_i P^i$ where $\lambda_i\ge 0$ for all $i$ and $\sum_{i = 1}^k\lambda_i = 1$.

To show that $G$ is $1$-extendable, we need to argue that every edge $ab\in E(G)$ is contained in a perfect matching.
Note that $\varphi(ab)>0$ (since otherwise $E(a)\setminus \{ab\}$ would be a set of edges of unit $\varphi$-weight
not equal to a star). Consequently,
$Q_{a,b} = \sum_{i = 1}^k\lambda_i P^i_{a,b}>0$, so there exists some $j\in \{1,\ldots, k\}$ such that
$\lambda_j >0$ and $P^j_{a,b}>0$. We claim that $M = \{xy\,:\,P^j_{x,y}>0\}$ is a perfect matching in $G$.
To see this, it suffices to show that $M\subseteq E(G)$. But this follows from the fact that
$xy\in M$ implies
$Q_{x,y} = \sum_{i = 1}^k\lambda_i P^i_{x,y}\ge \lambda_jP^j_{x,y}>0$.
We conclude that $G$ is $1$-extendable.
\end{proof}

Our second proof of Lemma~\ref{lem:equistarable-1-extendable} is derived using the following characterization of $k$-extendable bipartite graphs.

\begin{thm}[Plummer~\cite{Plum}]\label{bipartite-k-extendable}
Let $k\geq 1$ and let $G=(V,E)$ be a connected bipartite graph with a bipartition $\{A,B\}$ of its vertex set and $|V|\geq 2k$. Then, $G$ is $k$-extendable if and only if $|A|=|B|$ and for all non-empty subsets $X \subseteq A$ with $|X|\leq |A|- k$, it holds that $|N(X)|\geq |X|+ k$.
\end{thm}

\begin{sloppypar}
\begin{proof}[Second proof of Lemma \ref{lem:equistarable-1-extendable}]
Suppose for a contradiction that $G$ is equistarable but not $1$-extendable.
Let $A, B$ be a bipartition of $G$ and let $\varphi:E(G)\to \mathbb{R}_+$ be an equistarable weight function of $G$.
Since $\delta(G)\geq 2$, every star is maximal, and we have
$|A| = \varphi(E(A)) = \varphi(E(G)) = \varphi(E(B)) = |B|.$
On the other hand, since $G$ is not $1$-extendable, Theorem \ref{bipartite-k-extendable} implies that there exists
$X\subseteq A$ with $|X|\leq |A|-1$ such that $|N(X)|< |X|+1$.
Since $G$ is connected, there exists an edge $e\in E(G)$ connecting a vertex of $N(X)$ with a vertex in $A\setminus X$, leading us to the next chain of equalities and inequalities:
$$|X| = \varphi (E(X)) < \varphi (E(X)) + \varphi (e) \leq \varphi (E(N(X))) = |N(X)| \leq |X|,$$
a contradiction.
\end{proof}
\end{sloppypar}

In the proof of Lemma~\ref{lem:bipartite equistarable} below, we will make use of the following result on matchings.

\begin{thm}[Dulmage--Mendelsohn, see \cite{LovPlum} and \cite{BHR_Star}]\label{thm:matching extension}
Let G be a bipartite graph with bipartition $V=A\cup B$. Suppose $M_A$ and $M_B$ are matchings in $G$. Then there exists a matching $M \subseteq M_A \cup M_B$
which covers all the vertices of $A$ covered by $M_A$ and all the vertices of $B$ covered by $M_B$.
\end{thm}

\begin{lem}\label{lem:bipartite equistarable}
Every component of an equistarable bipartite graph is either a star or $2$-internally extendable.
\end{lem}

\begin{proof}
Let $G'$ be an equistarable bipartite graph and let $G=(V,E)$ be a component of $G'$. Then $G$ is equistarable by Lemma~\ref{lem:equistarable components}. Fix a bipartition $\{A,B\}$ of $V$, and let $\varphi:E \rightarrow \mathbb{R}^+$ be an equistarable weight function of $G$.

Suppose that $G$ is not a star.
Let $L$ be the set of leaves of $G$ and let us say that a vertex $v\in V$ is {\it internal} if it is neither a leaf nor adjacent to a leaf.

We split the proof into two cases.

\textbf{Case 1:}
$L = \emptyset$.

Since the graph $G$ does not have any leaves, Lemma~\ref{lem:equistarable-1-extendable} implies that $G$ is $1$-extendable.
We claim that $G$ is also $2$-extendable.
Suppose that is not the case, and let $\{ab,a'b'\}$ with $a,a'\in A$ and $b,b'\in B$ be
a $2$-matching not contained in any perfect matching.
By Hall's Theorem, there exists a subset $S\subseteq A\setminus\{a,a'\}$
such that $|N(S)\cap (B\setminus \{b,b'\})|<|S|$.
Since the graph is $1$-extendable, $S$ must have edges to both $b$ and $b'$,
and we must have $|N(S)\cap (B\setminus \{b,b'\})|=|S|-1$.

It follows that the sum of the characteristic vectors of the stars corresponding to $N(S)\cup\{b,b'\}$ minus the sum of the characteristic vectors of the stars of vertices in $S$ defines a subset $F$ of the edges of $G$ such that $\{ab,a'b'\}\subseteq F$; in particular, $F$ is not a star.
Since $|N(S)\cup\{b,b'\}|=|S|+1$, we have $\varphi(F)=1$.
This contradicts the fact that $\varphi$ is an equistarable weight function of $G$.
Thus in this case $G$ is $2$-extendable, which is equivalent to the condition that $G$ is $2$-internally extendable.

\textbf{Case 2:}
$L \neq \emptyset$.

We start by proving that for every subset $S$ of internal vertices in $A$ (or in $B$) we have $|N(S)|\geq |S|+2$.
Since each vertex in $S\cup N(S)$ is the center of a maximal star, we have
$|S| = \varphi(E(S)) \leq \varphi(E(N(S))) = |N(S)|$ and therefore we just have to rule out the cases when
$|N(S)|\in \{|S|,|S|+1\}$.

If $|N(S)|=|S|$, then since $G$ is connected,
and $S \neq A$ (since $L\neq \emptyset$ in this case), there must exist an edge incident with $N(S)$
but not with $S$. Therefore, since
$\varphi$ is strictly positive on all edges, we get $|S| = \varphi(E(S)) < \varphi(E(N(S))) = |N(S)|$, a contradiction.
Suppose now that $|N(S)|=|S|+1$ for some subset $S\subseteq A$ of internal vertices, and let
$F = E(N(S))\setminus E(S)$. Clearly, $\varphi(F) = 1$.
Since $\varphi$ is an equistarable weight function of $G$, the set $F$ is a maximal star.
Since every vertex in $N(S)$ incident with an edge in $F$ is also incident with an edge not in $F$,
the star $F$ cannot be rooted at a vertex in $N(S)$. Therefore, it is rooted at some vertex $x\in A\setminus S$
with $d_G(x)\ge 2$.
Since the star is maximal, we have $N(x) \subseteq N(S)$. By connectedness, this implies that
$A = S\cup\{x\}$ and $N(S) = B$.
In particular, all vertices of $A$ are of degree at least $2$. Therefore, $G$ has at least one leaf in the set $B=N(S)$,
which is in contradiction with the assumption that $S$ consists of internal vertices only.
Therefore we have shown that $|N(S)|\geq |S|+2$.

Now let $M$ be a $2$-matching of $G$.
The inequality we have proven implies that for every subset $S$ of internal vertices in $A-V(M)$ (or in $B-V(M)$), we have
$|N_{G-V(M)}(S)|\ge |S|$.
Therefore, by Hall's theorem there exists a matching $M_A$ in $G-V(M)$ covering all internal vertices in $A-V(M)$ and a matching $M_B$ in $G-V(M)$ covering all internal vertices in $B-V(M)$.
By Theorem \ref{thm:matching extension}, graph $G$ contains a matching $M' \subseteq M_A\cup M_B$ covering all internal vertices in $V(G)\setminus V(M)$.
This matching together with $M$ covers all internal vertices in $G$. Some vertices incident with a leaf might still be uncovered, but they can be covered one by one with new edges to form a perfect internal matching containing $M$.
This proves that $G$ is $2$-internally extendable and hence the lemma is proved.
\end{proof}

Now we have everything ready to prove the characterization of equistarable bipartite graphs,

\begin{thm}\label{thm:main_for_bipartite}
For every bipartite graph $G$ without isolated vertices, the following are equivalent:
\begin{description}
  \item[(a)] Every component of $G$ is either a star or $2$-internally extendable.
  \item[(b)] Every $2$-matching extends to a perfect internal matching.
  \item[(c)] $G$ is strongly equistarable.
  \item[(d)] $G$ is equistarable.
\end{description}
\end{thm}

\begin{proof}
Since bipartite graphs are triangle-free, we already know (a) $\Leftrightarrow$ (b) (by Lemma~\ref{lem:co-line-gpg})
and (a) $\Rightarrow$ (c) $\Rightarrow$ (d) (by Table~\ref{tabela}).
Lemma \ref{lem:bipartite equistarable} establishes the implication (d) $\Rightarrow$ (a)
completing the proof of the theorem.
\end{proof}

Table~\ref{tabela} and Theorem~\ref{thm:main_for_bipartite} imply the following.

\begin{cor}
Orlin's conjecture holds for complements of line graphs of bipartite graphs. That is,
within the class of complements of line graphs of bipartite graphs,
every equistable graph is a general partition graph.
\end{cor}

Lemma~\ref{lem:equistarable components} and its proof show that triangle-free equistarable graphs are not closed under taking disjoint union.
On the other hand, Theorem~\ref{thm:main_for_bipartite}, in particular the equivalence of items (a) and (d), shows that
for the case of bipartite graphs, we have the following.

\begin{cor}
A bipartite graph $G$ is equistarable if and only if every connected component of $G$ is equistarable.
\end{cor}

We conclude this section with an algorithmic remark. To the best of our knowledge, the computation complexity status of recognizing graphs in each of the following classes is open: (strongly) equistable graphs, general partition graphs, triangle graphs, (strongly) equistarable graphs.
The characterization of equistarable bipartite graphs given by Theorem~\ref{thm:main_for_bipartite}
implies an efficient recognition algorithm for this special case.

\begin{prop}\label{prop:bipartite_recognition}
Equistarable bipartite graphs can be recognized in polynomial time.
\end{prop}

\begin{proof}
Given a graph $G$, testing for bipartiteness can be done in linear time, and we may also assume that $G$ has no isolated vertices.
By Theorem \ref{thm:main_for_bipartite}, in order to determine if $G$ is equistarable, it is enough to enumerate all
of the $O(|E(G)|^2)$ $2$-matchings of $G$, and check, for each of them, if it can be extended to a perfect internal matching.
Given a $2$-matching $M$, this can be done as follows. Defining $U=\{u\in V(G)\setminus V(M) \mid d_G(u)>1\}$, the
 problem becomes equivalent to the problem of determining if the graph $G - V(M)$ contains a matching covering all vertices in $U$.
This problem can be solved in polynomial time either via matching matroids (see, e.g., \cite{schrijver-book}), or
by reducing the problem to an instance of the maximum weight matching problem.
This can be done by assigning a weight $w(e)=|e\cap U|\in \{0,1,2\}$ to each edge $e$ of the graph $G - V(M)$. In this case, the graph $G - V(M)$ contains a matching that covers $U$ if and only if the graph $G - V(M)$ contains a matching $M'$ of total weight $w(M')\geq |U|$. Since the maximum weight matching problem is solvable in polynomial time~\cite{Edmonds}, the proposition is proved.
\end{proof}

\section{Equistarable forests}\label{sec:forests}

Recall that a forest is an acyclic graph. In this section we show that the classes of $(i)$ forests in which each connected component is $2$-internally extendable or a star, $(ii)$ strongly equistarable forests, $(iii)$ equistarable forests, and $(iv)$ $P_5$-constrained forests,
all coincide (cf.~Table~\ref{tabela}).

\begin{lem}\label{lem:a1-ext. trees}
Every tree $T$ with at least one edge is $1$-internally extendable.
\end{lem}

\begin{proof}
The proof is by induction on the number of vertices.
For $n=2$ the statement is true, since the tree with two vertices has only one edge.
Let $n\geq 3$, let $T$ be a tree on $n$ vertices and let $e$ be the edge of $T$ that we want to extend into a perfect internal matching.
Suppose the statement of the lemma is true for every tree on at most $n-1$ vertices. We now show that it holds also for $T$.
The removal of the endpoints of $e$ (along with all the edges incident to them) from $T$ results in a forest with connected components $T_1, \dots T_k$, $k\geq 1$. Let $I =\{1\leq i\leq k \mid E(T_i)\neq \emptyset\}$.
For each $i\in {I}$, a tree $T_i$ has no more than $n-1$ vertices (and at least one edge) and is therefore $1$-internally extendable by the inductive hypothesis.
For each such $i$, let $v_i\in V(T_i)$ be the unique vertex in $T_i$ adjacent to an endpoint of $e$, and let $e_i$ be an edge in $T_i$ incident with $v_i$.
Let $M_i$ be a perfect internal matching of $T_i$ that extends $e_i$. By our inductive hypothesis, $M_i$ exists.
Using the fact that every leaf of $T$ is also a leaf of one of the subtrees, we can construct a perfect internal matching $M$ of $T$ containing $e$ as:
$M = \{e\} \cup \underset{i\in I}{\bigcup}M_i$. Therefore we conclude that $T$ is $1$-internally extendable.
\end{proof}

Using the above lemma, we now prove the stated characterization.

\begin{thm}\label{thm:main_for_forests}
For every forest $F$ without isolated vertices, the following are equivalent:
\begin{description}
  \item[(a)] Every connected component of $F$ is either a star or $2$-internally extendable.
  \item[(b)] Every $2$-matching extends to a perfect internal matching.
  \item[(c)] $F$ is strongly equistarable.
  \item[(d)] $F$ is equistarable.
  \item[(e)] $F$ is $P_5$-constrained.
\end{description}
\end{thm}

\begin{proof}
The fact that (a), (b), (c), (d) are equivalent follows from Theorem~\ref{thm:main_for_bipartite}.
Since forests are triangle-free, implication (d) $\Rightarrow$ (e) is given by Table~\ref{tabela}.

It remains to show (e) $\Rightarrow$ (a).
Let a forest $F$ be $P_5$-constrained. Let $T$ be a connected component of $F$. Then, $T$ is also $P_5$-constrained.
Suppose that $T$ is not a star.
Fix a $2$-matching $M=\{e,f\}$ of $T$ and consider the (unique) shortest path $P$ in $F$ between $e$ and $f$.
We construct another matching $M'$ by putting in it for every vertex of $P$ not
covered by $M$, an arbitrary edge incident with it and not in $P$.
(Since $T$ is $P_5$-constrained, all the vertices of $P$ have degree at least $3$.)
By deleting all edges in $P$, we are left with a forest $F'$ consisting of trees (note that every leaf of $T$ is also a leaf in $F$), each of which contains
exactly one edge of $M'\cup M$.
By Lemma \ref{lem:a1-ext. trees}, this edge can be extended to a perfect internal matching in the corresponding tree.
This means that the matching $M' \cup M$ can be extended to a perfect internal matching of $T$,
thus $F$ is $2$-internally extendable.
\end{proof}

Since forests are bipartite graphs, testing whether a given forest is equistarable can be done in polynomial time by Proposition \ref{prop:bipartite_recognition}. Theorem~\ref{thm:main_for_forests} implies that this can be done even in linear time.

\begin{prop}
Equistarable forests can be recognized in linear time.
\end{prop}

\begin{proof}
It is well know that acyclicity can be tested in linear time using depth-first search. Given a forest $F$, in order to determine if $F$ is equistarable, it is enough to check whether each connected component $C$ of $F$ is $P_5$-constrained (by Theorem \ref{thm:main_for_forests}).
This means that it is enough to identify all vertices $v$ with $d_F(v)=2$ and check that each of them has at least one leaf in their neighborhood.
\end{proof}


\smallskip
\subsection*{Acknowledgement}

The third author is grateful to Nicolas Trotignon and Denis Cornaz for stimulating discussions on the topic.

\bibliographystyle{abbrv}

\bibliography{equistarable-bib}{}

\end{document}